\documentclass[11pt]{article}
\usepackage{graphicx}
\usepackage{hhline}
\usepackage{ifxetex} 
\usepackage{setspace} 
\usepackage{amsmath,amsthm,amssymb,amsfonts}
\usepackage{amsmath}
         \usepackage{makeidx}
         \usepackage{ifxetex} 
\usepackage[all]{xy}
\usepackage[pagebackref=true,colorlinks,linkcolor=blue,citecolor=magenta]{hyperref}
\usepackage{fancybox}
\usepackage{fancyhdr}
\usepackage{tocbibind}
\usepackage{makeidx}
\makeindex
\makeatletter \oddsidemargin .4in \evensidemargin -.1in \textwidth
14.5cm \topmargin 0.1cm \textheight 22cm
\newcommand{\doublespacinggg}{\let\CS=\@currsize\renewcommand{\baselinestretch}{1.55}\tiny\CS}
\newcommand{\doublespacingg}{\let\CS=\@currsize\renewcommand{\baselinestretch}{1.75}\tiny\CS}
\doublespacinggg

\newtheorem{thm}{Theorem}[section]
 \newtheorem{cor}[thm]{Corollary}
 
   \newtheorem{examp}[thm]{Example}
  \newtheorem{defin}[thm]{Definition}
  \newtheorem{prop}[thm]{Proposition}

\newcommand{\be}{\begin{equation}}
\newcommand{\ee}{\end{equation}}
\newcommand{\bea}{\begin{eqnarray}}
\newcommand{\eea}{\end{eqnarray}}

\newcommand{\bee}{\begin{eqnarray*}}
\newcommand{\eee}{\end{eqnarray*}}

\title{On A Generalization of Weak Armendariz Rings}
\author{ \small Mahboubeh Sanaei $^{*}$,  
Shervin Sahebi $^{**}$ and Hamid H. S. Javadi $^{***}$\\ $^{*,**}$ 
Department of Mathematics, Islamic Azad University,\\
Central Tehran Branch, 13185/768, Iran, \\email: sahebi@iauctb.ac.ir; mah.sanaei.sci@iauctb.ac.ir
\\ 
$^{***}$  Department of Mathematics and Computer Science, Shahed University,\\ Tehran, Iran, email: h.s.javadi@shahed.ac.ir.\\}

\begin{document}

\date{}
\maketitle \noindent \vspace{-.8cm}

\doublespacingg

\begin{center}
\begin{minipage}{11cm} \footnotesize { \textsc{Abstract:}
We introduce the notion of J-Armendariz rings, 
which are a generalization of weak Armendariz rings and investigate their properties. 
We show that any local ring is J-Armendariz, and then find a local ring that is not weak Armendariz.
Moreover, we prove that a ring $ R $ is J-Armendariz if and only if the 
$ n $-by-$ n $ upper triangular matrix ring 
$ T_{n}(R) $
is J-Armendariz. 
For a ring $R$ and for some $e^{2}=e\in R$,
 we show that if $R $ is an abelian ring, 
then $R$ is J-Armendariz if and only if $eRe$ is J-Armendariz.
Also if the polynomial ring $ R[x] $ is J-Armendariz, 
then it is proven that the Laurent polynomial ring $R[x,x^{-1}]$ is J-Armendariz.
  }
 \end{minipage}
\end{center}

 \vspace*{.4cm}

 \noindent {\footnotesize {\bf Mathematics Subject Classification 2010: 16U20, 16S36,16W20}  \\
 {\bf Keywords:} Armendariz Ring; Weak Armendariz Ring; J-Armendariz Ring  }

\vspace*{.2cm}

\doublespacing

\section{Introduction} 
Throughout this article, $ R $ denotes an associative ring with identity. 
For a ring $R$, $ Nil(R) $ 
denotes the set of nilpotents elements in
 $ R $.  In 1997, Rege and Chhawchharia introduced the notion of an Armendariz ring. 
They called a ring $ R $ an Armendariz ring if whenever polynomials 
$ f(x)=a_{0}+a_{1}x+\cdots +a_{n}x^{n}$
 and $ g(x)=b_{0}+b_{1}x+\cdots +b_{m}x^{m} \in R[x]$
 satisfy $ f(x)g(x)=0 $ then $ a_{i}b_{j} =0$ for all $ i$ and $j $. 
The name "Armendariz ring" is chosen because Armendariz  
\cite[Lemma 1]{EP}
 has been shown that reduced ring (that is a ring without nonzero nilpotent) saisfies this condition. 
A number of properties of the Armendariz rings have been studied in
 \cite{a02, EP, javadi, a03, a04, Reg}. 
 So far Armendariz rings are generalized in several forms. 
 A generalization of Armendariz rings has been investigated in \cite{alpha}
Liu and Zhao \cite{zhao} called a ring $R$ weak Armendariz if whenever polynomials 
 $f(x)=a_{0}+a_{1}x+\cdots + a_{n}x^{n}$,
 $g(x)=b_{0}+b_{1}x+\cdots + b_{m}x^{m} \in R[x]$
satisfy $f(x)g(x)=0$, then 
$a_{i}b_{j}\in Nil(R)$
 for all $i$ and $j$.  
Recall that the Jacobson radical of a ring $ R $, 
is defined to be the intersection of all the maximal left ideals of $ R $. 
We use $ J(R) $ to denote the Jacobson radical of $ R $.
We call a ring $ R $, \textit{J-Armendariz} if 
whenever polynomials $ f(x)=a_{0}+a_{1}x+\cdots +a_{n}x^{n}$ 
and $ g(x)=b_{0}+b_{1}x+\cdots +b_{m}x^{m} \in R[x]$ satisfy 
$ f(x)g(x)=0 $ then $ a_{i}b_{j} \in J(R)$ for all $ i$ and $j $.
Clearly, weak Armendariz rings are J-Armendariz. Moreover, for an artinian ring, 
weak Armendariz rings and J-Armendariz rings are the same.
But, there exist a J-Armendariz ring that are not weak Armendariz.
Thus J-Armendariz rings are a proper generalization of weak Armendariz rings. 
Furthermore, we prove that the local rings are J-Armendariz. 
Then we give an example to show that Local rings are not weak Armendariz in general.
\section{J-Armendariz rings}
 \noindent
In this section J-Armendariz rings are introduced as a generalization of weak Armendariz ring. 
\begin{defin}\label{def 1.1}  
A ring $ R $ is said to be J-Armendariz if for any nonzero polynomial  $ f(x)=\sum_{i=0}^{n}a_{i}x^{i}$ and $g(x)= \sum_{j=0}^{m}b_{j}x^{j} \in R[x] ,  $ $ f(x)g(x)=0 $, implies that
 $a_{i}b_{j} \in J(R)$ for each $i,j $.
\end{defin}
Clearly, any Armendariz ring and weak Armendariz ring is J-Armendariz.
In the following, we will see that the J-Armendariz rings are not nescessary weak Armendariz.
\begin{examp}
 Let $ A $ be the 3 by 3 full matrix ring over the power series ring $ F[[t]] $ over a field $ F $. Let
\begin{center}
$ B=\lbrace M=(m_{ij}) \in A | m_{ij}\in tF[[t]] ~for~ 1\leq i,j\leq 2 ~ and~ m_{ij}=0 ~for~ i=3 ~or~ j=3\rbrace$ 
\end{center}
\begin{center}
$ C=\lbrace M=(m_{ij}) \in A | m_{ii}\in F ~ and~ m_{ij}=0 ~for~ i\neq j\rbrace $. 
\end{center}
Let $ R $ be the subring of $ A $ generated by $ B $ and $ C $. Let $ F=\mathbb{Z}_{2} $.
Note that every element of $ R $ is of the form
$ \Bigl(\begin{smallmatrix}
  a +f_{1}& f_{2} & 0 \\
  f_{3} & a+f_{4} & 0 \\
  0 &0 & a
\end{smallmatrix}\Bigr)$ for some $ a\in F $ and $ f_{i} \in tF[[t]]~(i=1,2,3,4)$ and $ J(R)= tR $.
Let
\begin{center}
$ f(x)=\sum_{i=0}^{n}\Bigl(\begin{smallmatrix}
  a_{i} +f_{1_{i}}& f_{2_{i}} & 0 \\
  f_{3_{i}} & a_{i}+f_{4_{i}} & 0 \\
  0 &0 & a_{i}
\end{smallmatrix} \Bigr)x^{i} $
and
$ g(x)=\sum_{j=0}^{m}\Bigl(\begin{smallmatrix}
  b_{j} +g_{1_{j}}& g_{2_{j}} & 0 \\
  g_{3_{j}} & b_{i}+g_{4_{j}} & 0 \\
  0 &0 & b_{j}
\end{smallmatrix}\Bigr)x^{j} \in R[x] $.
\end{center} Assume that 
$ f(x)g(x)=0 $. Then $ a_{i}b_{j}=0 $ for all $ i $ and $ j $ and so
\begin{center}
$\Bigl(\begin{smallmatrix}
  a_{i} +f_{1_{i}}& f_{2_{i}} & 0 \\
  f_{3_{i}} & a_{i}+f_{4_{i}} & 0 \\
  0 &0 & a_{i}
\end{smallmatrix}\Bigr)\Bigl(\begin{smallmatrix}
  b_{j} +g_{1_{j}}& g_{2_{j}} & 0 \\
  g_{3_{j}} & b_{i}+g_{4_{j}} & 0 \\
  0 &0 & b_{j}
  \end{smallmatrix}\Bigr)\in tR $.
\end{center} 
Hence $ R $ is J-Armendariz.
Now consider two polynomials over $ R $
\begin{center}
$ f(x)=te_{11}+ te_{12}x+te_{21}x^{2}+te_{22}x^{3}$, 
$ g(x)= -t(e_{21}+e_{22})+ t(e_{11}+e_{12})x$. 
\end{center}
Then $ f(x)g(x)=0 $, but $te_{11}t(e_{21}+e_{22})\notin Nil(R) $, 
and so the ring $ R $ is not weak Armendariz. 
\end{examp}
\begin{prop}\label{proposition1}
Let $ R $ be a ring and $ I  $ an ideal of $ R $ such that $ R/I $ is J-Armendariz. 
If $ I\subseteq J(R) $, then $ R $ is J-Armendariz.
\end{prop}
\begin{proof}
Suppose that
$ f(x)=a_{0}+a_{1}x+a_{2}x^{2}+\cdots +a_{n}x^{n} $ and 
$ g(x)=b_{0}+b_{1}x+b_{2}x^{2}+\cdots +b_{m}x^{m}  $ are polynomials in
$ R[x] $
 such that 
$ f(x)g(x)=0 $.
This implies
\begin{center}
$ (\bar{a_{0}}+\bar{a_{1}}x+\bar{a_{2}}x^{2}+\cdots +\bar{a_{n}}x^{n})
(\bar{b_{0}}+\bar{b_{1}}x+\bar{b_{2}}x^{2}+\cdots +\bar{b_{m}}x^{m})=\bar{0} $, 
\end{center}
in $ R/I $.
Thus
$ \bar{a_{i}}\bar{b_{j}} \in J(R/I)$, 
And so
$ a_{i} b_{j}\in J(R)$. This means that 
$ R $
is a J-Armendariz ring.
\end{proof}
\begin{cor}\label{corrolary1}
Let $ R $ be any local ring. Then $ R $ is J-Armendarz.
\end{cor}
One may ask if local rings are weak Armendariz, but the following gives a negative answer.
\begin{examp}\label{examp1}
 Let $ F $ be a field, $R= M_{2}(F) $ and $ R_{1}=R[[t]] $. Consider the ring
\begin{center}
  $ S =\lbrace \sum_{i=0}^{\infty}a_{i}t^{i}\in R_{1}| a_{0}\in  kI~ for ~k\in F   \rbrace$, 
\end{center}
where $ I $ is the identity matrix over $ F $.
It is obvious that $ S $ is local and so is J-Armendariz. Now for
$ f(x)=e_{11}t-e_{12}tx $
and
$ g(x)=e_{21}t +e_{11}tx \in S[x]$, we have 
$ f(x)g(x)=0 $, but 
$ (e_{11}t)^{2} $ is not nilpotent in $ S $, and so $ S $ is not weak Armendariz.
\end{examp}
\begin{thm}\label{theorem0}
Let $R_{t} $ be a ring, for each $ t\in I $. Then any direct product of rings
$\prod_{t\in I}R_{t} $, is J-Armendariz if and only if any $ R_{t} $ is J-Armendariz. 
\end{thm}\label{theorem1}
\begin{proof}
Suppose that
 $R_{t}  $ is J-Armendariz, for each $ t\in I  $ and 
$ R=\prod_{t\in I}R_{t} $. Let 
$ f(x)g(x)=0 $ for some polynomials
$ f(x)=a_{0}+a_{1}x+a_{2}x^{2}+\cdots +a_{n}x^{n} $, 
$ g(x)=b_{0}+b_{1}x+b_{2}x^{2}+\cdots +b_{m}x^{m} \in R[x] $,
where 
$ a_{i} =(a_{i_{1}},a_{i_{2}},\cdots , a_{i_{t}},\cdots)$,
$ b_{j} =(b_{j_{1}},b_{j_{2}},\cdots, b_{j_{t}},\cdots ) $
are elements of the product ring 
$ R $ for 
$ 1\leq i\leq n $ and 
$ 1\leq j \leq m $.
Define
\begin{center}
$ f_{t}(x)=a_{0_{t}}+a_{1_{t}}x+a_{2_{t}}x^{2}+\cdots +a_{n_{t}}x^{n} $,
$g_{t}(x)=b_{0_{t}}+b_{1_{t}}x+b_{2_{t}}x^{2}+\cdots +b_{m_{t}}x^{m}  $.
\end{center}
From 
$ f(x)g(x)=0 $,
we have 
$ a_{0}b_{0}=0, a_{0}b_{1}+a_{1}b_{0}=0,\cdots a_{n}b_{m}=0$, and this implies
  \begin{eqnarray*}
 a_{0_{1}}b_{0_{1}}=a_{0_{2}}b_{0_{2}}=\cdots =a_{0_{t}}b_{0_{t}}=\cdots =0 \\
a_{0_{1}}b_{1_{1}}+a_{1_{1}}b_{0_{1}}=a_{0_{2}}b_{1_{2}}+a_{1_{2}}b_{0_{2}}=\cdots =a_{0_{t}}b_{1_{t}}+a_{1_{t}}b_{0_{t}}=\cdots=0 \\
  a_{n_{1}}b_{m_{1}}= a_{n_{2}}b_{m_{2}}=\cdots=a_{n_{t}}b_{m_{t}}=\cdots =0
 \end{eqnarray*}
This means that 
 $ f_{t}(x)g_{t}(x) =0$ in 
 $ R_{t}[x] $, for each
 $ t\in I $. Since
 $ R_{t} $ is J-Armendariz for each $ t\in I $, then 
 $ a_{i_{t}}b_{j_{t}}\in J(R_{t}) $. Now the equation 
 $  \prod_{t\in I}J(R_{t})= J(\prod_{t\in I}R_{t})$, implies that 
 $ a_{i}b_{j} \in J(R) $, and so $ R  $ is J-Armendariz.
 Conversely, assume that 
$ R=\prod_{t\in I} R_{t}$ is J-Armendariz and
$ f_{t}(x)g_{t}(x) =0$
for some polynomials
$ f_{t}(x)=a_{0_{t}}+a_{1_{t}}x+a_{2_{t}}x^{2}+\cdots +a_{n_{t}}x^{n} $,
$g_{t}(x)=b_{0_{t}}+b_{1_{t}}x+b_{2_{t}}x^{2}+\cdots +b_{m_{t}}x^{m} \in R_{t}[x] $, 
with $ t\in I $.
Define 
$ F(x)=a_{0}+a_{1}x+a_{2}x^{2}+\cdots +a_{n}x^{n} $, 
$ G(x)=b_{0}+b_{1}x+b_{2}x^{2}+\cdots +b_{m}x^{m} \in R[x] $, where
$ a_{i} =(0,\cdots ,0, a_{i_{t}},0,\cdots)$,
$ b_{j} =(0,\cdots,0, b_{j_{t}},0,\cdots) \in R$. Since 
$ f_{t}(x)g_{t}(x)=0 $, we have 
$ F(x)G(x)=0 $. $ R $ is J-Armendariz, so
$ a_{i}b_{j}\in J(R) $. Therefore
$ a_{i_{t}}b_{j_{t}}\in J(R_{t}) $ and so $ R_{t} $ is J-Armendariz for each $ t\in I $.
 \end{proof}
The following example shows that for an Armendariz ring $R  $, 
every full $ n $-by-$ n $ matrix ring $ M_{n}(R) $ over $ R $ need not to be J-Armendariz.
\begin{examp}\label{example2}
Let 
$ F $ be a field and  
$ R=M_{2}(F) $. If
$ f(x)= \bigl( \begin{smallmatrix}
  0&1\\ 0&0
\end{smallmatrix} \bigr)-\bigl( \begin{smallmatrix}
  1&0\\ 0&0
\end{smallmatrix} \bigr) x $
and 
$ g(x)= \bigl( \begin{smallmatrix}
  1&1\\ 0&0
\end{smallmatrix} \bigr)+\bigl( \begin{smallmatrix}
  0&0\\ -1&-1
\end{smallmatrix} \bigr) x $,
then
$ f(x)g(x)=0 $. But 
$  \bigl( \begin{smallmatrix}
  1&0\\ 0&0
\end{smallmatrix} \bigr)\bigl( \begin{smallmatrix}
  1&1\\ 0&0
\end{smallmatrix} \bigr)=\bigl( \begin{smallmatrix}
  1&1\\ 0&0
\end{smallmatrix} \bigr)$ 
is not in $ J(R) $. Thus $ R $ is not J-Armendariz.
\end{examp}
Let $ R $ and $ S $ be two rings and $ M $ be an $ (R,S) $-bimodule. 
This means that $ M $ is a left $ R $-module and a right $ S $-module such that 
$ (rm)s=r(ms) $
for all $ r\in R $, $ m\in M $, and $ s\in S $. Given such a bimodule $ M $ we can form 
\begin{center}
$  T =\bigl( 
\begin{smallmatrix}
  R&M\\ 0&S
\end{smallmatrix} \bigr)=\big\lbrace \bigl( \begin{smallmatrix}r&m\\ 0&s \end{smallmatrix} \bigr)
 :  r\in R , m\in M , s\in S \big\rbrace$
\end{center} 
and define a multiplication on $ T  $ by using formal matrix multiplication:
\begin{center}
$ \bigl(
 \begin{smallmatrix}
  r&m\\ 0&s
\end{smallmatrix} \bigr)\bigl(
 \begin{smallmatrix}
  r^{\prime}&m^{\prime}\\ 0&s^{\prime}
\end{smallmatrix} \bigr)=\bigl(
 \begin{smallmatrix}
  rr^{\prime}&rm^{\prime}+ms^{\prime}\\ 0&ss^{\prime}
\end{smallmatrix} \bigr). $
\end{center}
This ring construction is called triangular ring $ T $. 
\begin{prop}\label{proposition2}
Let $ R $ and $ S $ be two rings and $ T $ be the triangular ring 
$
T=
\bigl( \begin{smallmatrix}
  R&M\\ 0&S
\end{smallmatrix} \bigr) 
$
(where $ M $ is an $ (R,S) $-bimodule). 
Then the rings $ R $ and $ S $ are J-Armendariz if and only if $ T $ is J-Armendariz. 
\end{prop}
\begin{proof}
Let
$ R $ and
$ S $
be J-Armendarz, and 
\begin{center}
$ f(x)=\bigl( \begin{smallmatrix}
  r_{0}&m_{0}\\ 0&s_{0}
\end{smallmatrix} \bigr)  +\bigl( \begin{smallmatrix}
  r_{1}&m_{1}\\ 0&s_{1}
\end{smallmatrix} \bigr)x+\cdots +\bigl( \begin{smallmatrix}
  r_{n}&m_{n}\\ 0&s_{n}
\end{smallmatrix} \bigr)x^{n}  $,
\end{center}
\begin{center}
$ g(x)=\bigl( \begin{smallmatrix}
  r^{\prime}_{0}&m^{\prime}_{0}\\ 0&s^{\prime}_{0}
\end{smallmatrix} \bigr)  +\bigl( \begin{smallmatrix}
  r^{\prime}_{1}&m^{\prime}_{1}\\ 0&s^{\prime}_{1}
\end{smallmatrix} \bigr)x+\cdots +\bigl( \begin{smallmatrix}
  r^{\prime}_{m}&m^{\prime}_{m}\\ 0&s^{\prime}_{m}
\end{smallmatrix} \bigr)x^{m}  \in T[x]$
\end{center} satisfy 
$ f(x)g(x)=0 $.
Define
\begin{center}
$ f_{r}(x) =r_{0}+r_{1}x+\cdots + r_{n}x^{n}$,
$  g_{r}(x) =r^{\prime}_{0}+r^{\prime}_{1}x+\cdots + r^{\prime}_{m}x^{m}\in R[x] $
\end{center}
 and
\begin{center}
$ f_{s}(x) =s_{0}+s_{1}x+\cdots + s_{n}x^{n}  $,
$  g_{s}(x) =s^{\prime}_{0}+s^{\prime}_{1}x+\cdots + s^{\prime}_{m}x^{m} \in S[x]. $
\end{center}
From 
$ f(x)g(x)=0 $, we have 
$ f_{r}(x)g_{r}(x)= f_{s}(x)g_{s}(x)=0 $, and since $ R $ and $ S $ are J-Armendariz then
$ r_{i}r^{\prime}_{j}\in J(R) $ and $  s_{i}s^{\prime}_{j}\in J(S) $ for each 
$ 1\leq i \leq n $, $ 1\leq j \leq m $.
Now from the fact  
$ J(T)= \bigl( \begin{smallmatrix}
  J(R)&M\\ 0&J(S)
\end{smallmatrix} \bigr) $, we obtain that 
$ \bigl( \begin{smallmatrix}
  r_{i}&m_{i}\\ 0&s_{i}
\end{smallmatrix} \bigr)\bigl( \begin{smallmatrix}
  r^{\prime}_{j}&m^{\prime}_{j}\\ 0&s^{\prime}_{j}
\end{smallmatrix} \bigr)\in J(T) $ for any
$ i,j $. Hence 
$ T $
is a J-Armendariz ring.
Conversely, let $ T $ be a J-Armendariz ring,
$ f_{r}(x) =r_{0}+r_{1}x+\cdots + r_{n}x^{n}$,
$  g_{r}(x) =r^{\prime}_{0}+r^{\prime}_{1}x+\cdots + r^{\prime}_{m}x^{m}\in R[x] $, such that
$ f_{r}(x)g_{r}(x)=0 $, and 
$ f_{s}(x) =s_{0}+s_{1}x+\cdots + s_{n}x^{n}  $,
$  g_{s}(x) =s^{\prime}_{0}+s^{\prime}_{1}x+\cdots + s^{\prime}_{m}x^{m} \in S[x] $, such that
$ f_{s}(x)g_{s}(x)=0 $.
If
\begin{center}
$ f(x)=\bigl( \begin{smallmatrix}
  r_{0}&0\\ 0&s_{0}
\end{smallmatrix} \bigr)  +\bigl( \begin{smallmatrix}
  r_{1}&0\\ 0&s_{1}
\end{smallmatrix} \bigr)x+\cdots +\bigl( \begin{smallmatrix}
  r_{n}&0\\ 0&s_{n}
\end{smallmatrix} \bigr)x^{n}  $
and
$ g(x)=\bigl( \begin{smallmatrix}
  r^{\prime}_{0}&0\\ 0&s^{\prime}_{0}
\end{smallmatrix} \bigr)  +\bigl( \begin{smallmatrix}
  r^{\prime}_{1}&0\\ 0&s^{\prime}_{1}
\end{smallmatrix} \bigr)x+\cdots +\bigl( \begin{smallmatrix}
  r^{\prime}_{m}&0\\ 0&s^{\prime}_{m}
\end{smallmatrix} \bigr)x^{m}  \in T[x]$
\end{center}
Then from
$ f_{r}(x)g_{r}(x)=0  $ and $ f_{s}(x)g_{s}(x)=0  $ it follows that
$ f(x)g(x)=0 $.
Since 
$ T $ is a J-Armendariz ring, 
$ \bigl( \begin{smallmatrix}
  r_{i}&0\\ 0&s_{i}
\end{smallmatrix} \bigr)\bigl( \begin{smallmatrix}
  r^{\prime}_{j}&0\\ 0&s^{\prime}_{j}
\end{smallmatrix} \bigr)\in J (T)=\bigl( \begin{smallmatrix}
  J(R)&0\\ 0&J(S)
\end{smallmatrix} \bigr) $.
Thus
$ r_{i} r^{\prime}_{j}\in J(R)$ and
$ s_{i} s^{\prime}_{j}\in J(S)$ for any $ i,j $. This shows that 
$ R $ and $ S $ are J-Armendariz.
\end{proof}
Given a ring $ R $ and a bimodule $ _{R}M_{R} $, the trivial extension of $ R $ by $ M $ is the ring
 $ T(R,M) =R\bigoplus M$
with the usual addition and the multiplication
$$(r_{1},m_{1})(r_{2},m_{2})=(r_{1}r_{2},r_{1}m_{2}+m_{1}r_{2}).  $$
This is isomorphic to the ring of all matrices 
$
\bigl( \begin{smallmatrix}
  r&m\\ 0&r
\end{smallmatrix} \bigr) 
$,
where 
$ r\in R $
and 
$ m\in M $
and the usual matrix operations are used.
\begin{cor}\label{corollary2}
A ring $ R $ is J-Armendariz if and only if the trivial extension $ T(R,R) $ is a J-Armendariz ring.
\end{cor}
\begin{cor}\label{corollary3}
A ring $ R $ is J-Armendariz if and only if, for any $ n $, $ T_{n}(R) $ is J-Armendariz. 
\end{cor}
\begin{cor}\label{corollary4}
If $ R $ is a Armendariz ring then, for any $ n $, $ T_{n}(R) $ is a J-Armendariz ring. 
\end{cor}
Recall that a ring $ R $ is said to be \textit{abelian} if every idempotent of it is central.
Armendariz rings are abelian \cite[Lemma 7]{a04},
but the next example shows that weak Armendariz and J-Armendariz rings need not to be abelian in general.

\begin{examp}\label{abelian}
Let $ F $ be a field. By Corollary \ref{corollary4}, $ R=T_{2}(F) $ is a J-Armendariz ring. We see that
$ \bigl( \begin{smallmatrix}
  0&0\\ 0&1
\end{smallmatrix} \bigr) $ is an idempotent element in $ R $, that is not central. So $ R $ is not an abelian ring.
\end{examp}

\begin{prop}\label{proposition4}
Let $ R $ be a J-Armendariz ring. Then for any idempotent $ e $ of $ R $, $ eRe $ is J-Armendariz. 
The converse holds if $ R $ is an abelian ring.
\end{prop}
\begin{proof}
Let
$ f(x)=\sum_{i=0}^{n}a_{i}x^{i} $, 
 $ g(x)=\sum_{j=0}^{m}b_{j}x^{j} \in (eRe)[x] $
be such that 
$ f(x)g(x)=0 $. Since $ R $ is J-Armendariz and 
$ a_{i}, b_{j} \in eRe\subseteq R $, then we have
$ a_{i}b_{j} \in J(R)\cap eRe= J(eRe) $. This means that 
$ eRe $ is J-Armendariz. Conversely, let $ eRe $ be a J-Armendariz ring and
$ f(x)=\sum_{i=0}^{n}a_{i}x^{i} $, 
 $ g(x)=\sum_{j=0}^{m}b_{j}x^{j} \in R[x] $,
such that 
$ f(x)g(x)=0 $.
By the hypothesis, 
$ 0=ef(x)eg(x)e\in (eRe)[x] $, and since $ eRe $ is J-Armendariz, we have 
$ a_{i}b_{j}\in J(eRe)=J(R)\cap eRe $. Thus $ R $ is J-Armendariz.
\end{proof}

In \cite{a02} it is proven that a ring $ R $ is Armendariz if and only if its polynomial ring $ R[x] $ is Armendariz. 
More generally, we can get the following result.
\begin{thm}\label{theorem3}
If the ring $ R[x] $ is J-Armendariz, then $ R $ is J-Armendariz. The converse holds if  
$ J(R)[x]\subseteq J(R[x]) $.
\end{thm}
\begin{proof}
Suppose that 
$ R[x] $
is a J-Armendariz ring. Let 
$ f(y)=\sum_{i=0} ^{n}a_{i}y^{i}$ and $ g(y)=\sum_{j=0} ^{m}b_{j}y^{j} $ be nonzero plynomials  $ \in R[y] $, such that 
$ f(y)g(y)=0 $. Since
$ R[x] $ is J-Armendariz and
$ R\subseteq R[x] $, we have
$ a_{i}b_{j} \in R \cap J(R[x]) $, and so 
$ R $ is J-Armendariz.
Conversely, suppose that 
$ R $ is J-Armendariz and 
$ J(R )[x] \subseteq J(R[x])$.
Let
$ F(y)=f_{0}+f_{1}y+\cdots + f_{n}y^{n} $ and 
$ G(y)=g_{0}+g_{1}y+\cdots + g_{m}y^{m}$ be polynomials in 
$ R[x][y]$, with
$ F(y)G(y)=0 $. We also let 
$  f_{i}(x)=a_{i_{0}}+a_{i_{1}}x+a_{i_{2}}x^{2}+\cdots + a_{i_{\omega_{i}}}x^{\omega_{i}}$ and
$g_{j}(x)=b_{j_{0}}+b_{j_{1}}x+b_{j_{2}}x^{2}+\cdots + b_{j_{\nu_{j}}}x^{\nu_{i}} \in R[x]  $ for each
$ 0\leq i \leq n $ and $ 0\leq j \leq m $ . Take a positive integer 
$ t $ suhc that 
$ t\geq deg(f_{0}(x))+ deg(f_{1}(x))+\cdots + deg(f_{n}(x))+deg(g_{0}(x))+deg(g_{1}(x))+\cdots +deg(g_{m}(x))$, 
where the degree is as polynomials in $ x$ and the degree of zero polynomial is taken to be $ 0 $.  
Then
$ F(x^{t})=f_{0}+f_{1}x^{t}+\cdots + f_{n}x^{tn} $ and 
$ G(x^{t})=g_{0}+g_{1}x^{t}+\cdots + g_{m}x^{tm} \in R[x]$
and the set of coefficients of the 
$ f_{i} $'s (resp. $g_{j}$'s) equals the set of coefficients of the
$F(x^{t}) $ (resp. 
$G(x^{t})  $). 
Since 
$ F(y)G(y)=0 $, then
$ F(x^{t})G(x^{t}) =0$. 
So 
$ a_{is_{i}}b_{jr_{j}}\in J(R) $, where
$ 0\leq s_{i}\leq \omega_{i} $, 
$ 0\leq r_{j}\leq \nu_{j} $.
By hypothesis we have
$ J(R)[x]\subseteq J(R[x]) $, and so
$ f_{i}g_{j}\in J(R[x]) $. It implies that
$ R $ is J-Armendariz.
\end{proof}
\begin{prop}\label{prposition3}
Let $ R $ is a J-Armendariz ring and $ S $
denotes a multiplicatively closed subset of a ring 
$ R $ consisting of central regular elements.  Let
$ S^{-1}R $ denotes the localization of 
$ R $ at $ S $. Then $ S^{-1}R $ is a J-Armendariz ring.
\end{prop}
\begin{proof}
Suppose that 
$ R $ is a J-Armendariz ring. Let
$ F(x)=\sum_{i=0}^{n} (\alpha_{i})x^{i}$ and
$ G(x)=\sum_{j=0}^{m}(\beta_{j}) x^{j}$ be nonzero polynomials in $ (S^{-1}R)[x] $ such that
$F(x)G(x)=0 $, where 
$ \alpha_{i}=a_{i}u^{-1} $,
$ \beta_{j}=b_{j}v^{-1} $, with 
$ a_{i},b_{j} \in R $ and 
$ u,v\in S $.
Since 
$ S $
is contained in the center of 
$ R $, we have 
$ F(x)G(x)= (uv)^{-1}(a_{0}+a_{1}x+a_{2}x^{2}+\cdots + a_{n}x^{n})(b_{0}+b_{1}x+b_{2}x^{2}+\cdots + b_{m}x^{m}) =0$.
Let
$  f(x)=a_{0}+a_{1}x+a_{2}x^{2}+\cdots + a_{n}x^{n}$ and
$g(x)=b_{0}+b_{1}x+b_{2}x^{2}+\cdots + b_{m}x^{m}  $. Then
$ f(x) $
and
$ g(x) $
are nonzero polynomials in
$ R[x] $ with
$ f(x)g(x)=0 $.
Since 
$ R $
is J-Armendariz, then
$ a_{i}b_{j}\in J(R) $. 
It means that 
$ \alpha_{i}\beta_{j}\in J(S^{-1}R) $, concluding that 
$ S^{-1}R $
is J-Armendariz.
\end{proof}
\begin{cor}
For a ring $ R $, if $ R[x] $ is J-Armendariz, Then $ R[x,x^{-1}] $ is J-Armendariz.
\end{cor}
\begin{proof}
Let $ S=\lbrace 1,x,x^{2},x^{3},x^{4},\ldots\rbrace $. 
Then $ S $ is a multiplicatively closed subset of $ R[x] $ consisting of central regular elements. 
Then the proof follows from Proposition \ref{prposition3}.
\end{proof}
  \singlespacing

\end{document}